\documentclass[a4paper,12pt,reqno]{amsart}

\usepackage{amsmath}
\usepackage{amssymb}
\usepackage{amsfonts}
\usepackage{mathrsfs}
\usepackage{a4wide}
\usepackage{graphicx,subfigure}
\newcommand{\figdir}{./figures}


\newtheorem{statement}{}
\newtheorem{theo}[statement]{Theorem}
\newtheorem{coro}[statement]{Corollary}
\newtheorem{proposition}[statement]{Proposition}

\newtheorem{remq}[statement]{Remark}
\newtheorem{lem}[statement]{Lemma}

\newcommand{\DD}{\mathbb{D}}
\newcommand{\T}{\mathbb{T}}

\renewcommand{\Im}{\mathrm{Im }}

\newcommand\CC{\mathbb{C}}

\newcommand\TT{\mathbb{T}}

\newcommand\RR{\mathbb{R}}

\DeclareMathOperator{\dist}{dist}

\begin{document}

\bibliographystyle{plain}

\title{Cyclicity in weighted Bergman type spaces}

\author[Borichev]{A. Borichev}
\address{Borichev: LATP\\ Aix-Marseille Universit\'e\\39 rue F. Joliot-Curie\\ 13453 Marseille
\\France}
\email{borichev@cmi.univ-mrs.fr}

\author[El-Fallah]{O. El-Fallah}
\address{El-Fallah: Laboratoire Analyse et Applications -- URAC/03. D\'{e}partement de Math\'ema\-tiques,  Universit\'e Mohammed V, Rabat--Agdal--B.P. 1014 Rabat, Morocco}
\email{elfallah@fsr.ac.ma}

\author[Hanine]{A. Hanine}
\address{Hanine: Laboratoire Analyse et Applications -- URAC/03. D\'{e}partement de Math\'emati\-ques,  Universit\'e Mohammed V, Rabat--Agdal--B.P. 1014 Rabat, Morocco}
\address{LATP\\ Aix-Marseille Universit\'e\\39 rue F. Joliot-Curie\\ 13453 Marseille
\\France}
\email{hanine@cmi.univ-mrs.fr, abhanine@gmail.com}

\date{ \today}


\thanks{The research of the first author was partially supported by
the ANR FRAB. The research of the second and the third authors was partially supported by ``Hassan II Academy of Science and Technology". The research of the authors was partially supported by 
the Egide Volubilis program.}

\begin{abstract}
We use the so called resolvent transform method to study the cyclicity of 
the one point mass singular inner function in weighted Bergman type spaces.  
\end{abstract}
\maketitle

\section{Introduction.}

Let $X$ be a space of analytic functions in the open unit disc $\mathbb{D}$ 
of the complex plane $\CC$. Assume that $X$ is invariant with respect to the shift operator $M_{z}$ defined by $M_{z}(f) = zf$, $f\in X$. Given a function $f\in X$, the closed linear span in $X$ of all polynomial multiples of $f$ is denoted by $[f]$, it is 
an $M_{z}$-invariant subspace of $X$; is also the smallest closed $M_{z}$-invariant subspace of $X$ containing $f$, namely 
$$
[f]=\overline{\{pf : p \text{ is a polynomial}\}}.
$$
An element $f$ of $X$ is said to be cyclic (weakly invertible) in $X$ if $[f]= X$. 
If the polynomials are dense in $X$, an equivalent condition is that $1\in[f]$.

Given a positive non-increasing continuous function $\Lambda$ on $(0,1]$ and 
$E\subset\T=\partial\DD$, 
we denote by $\mathcal{B}_{\Lambda,E}^\infty$ the space of all analytic functions $f$ on $\DD$ such that
$$
\|f\|_{\Lambda,E,\infty}=\sup_{z\in\DD}|f(z)|e^{-\Lambda(\dist(z,E))}<+\infty,
$$
and by $\mathcal{B}_{\Lambda,E}^{\infty,0}$ its separable subspace 
$$
\mathcal{B}_{\Lambda,E}^{\infty,0}=\bigl\{f\in \mathcal{B}_{\Lambda,E}^\infty:
\lim_{\dist(z,E)\to 0}|f(z)|e^{-\Lambda(\dist(z,E))}=0\bigr\}.
$$
Analogously, integrating with respect to area measure on the disc, we define the spaces 
$\mathcal{B}_{\Lambda,E}^p$, $1\le p<\infty$: 
$$
\mathcal{B}_{\Lambda,E}^p=\Bigl\{f\in\textrm{Hol}(\DD):
\|f\|^p_{\Lambda,E,p}=\int_{\DD}|f(z)|^p e^{-p\Lambda(\dist(z,E))}<+\infty\Bigr\}.
$$

If $E=\T$, we use the notation 
$\mathcal{B}_{\Lambda}^\infty$, $\mathcal{B}_{\Lambda}^{\infty,0}$, 
$\mathcal{B}_{\Lambda}^p$. Let us remark that either $\Lambda(0^{+})=+\infty$ or 
$\mathcal{B}_{\Lambda}^\infty=H^\infty$, $\mathcal{B}_{\Lambda}^{\infty,0}=\{0\}$,
$\mathcal{B}_{\Lambda}^p=\mathcal{B}_{0}^p$.

Given a sequence of positive numbers $w=(w_n)$ such that $|\log w_n|=o(n)$, $n\to\infty$,  
we define by $H^2_w$ the space of functions $f$ analytic in the unit disc such that 
$$
f(z)=\sum_{n\ge 0}a_nz^n,\qquad \sum_{n\ge 0}\frac{|a_n|^2}{w_n}<\infty.
$$
It is known that for log-convex sequences $w$ such spaces coincide with 
$\mathcal{B}_{\Lambda,\T}^2$ spaces for $\Lambda$ defined by $w$ (see \cite{BH}). 

The questions on cyclicity in the large Bergman type spaces $\mathcal{B}_{\Lambda}^p$, 
$1\le p<\infty$, go back to Carleman and Keldysh. In particular, Keldysh \cite{KE} proved in 1945 
that the singular function with one point singular mass $S(z)=e^{-\frac{1+z}{1-z}}$ 
is not cyclic in $\mathcal{B}_{1}^2$. Beurling \cite{B} studied cyclicity of $S$ in inductive limits of $H^2_w$ ($\mathcal{B}_{\Lambda}^2$) spaces. In 1974 Nikolski 
\cite[Section 2.6]{NN} proved that if $\liminf_{t\to 0}\frac{\Lambda(t)}{\log1/t}>0$, 
and $S$ is cyclic in $\mathcal{B}_\Lambda ^\infty$, then
\begin{equation}
\int_0\sqrt{\frac{\Lambda(t)}{t}}dt=\infty.
\label{y9}
\end{equation}
In the opposite direction, he proved that if  
\begin{equation}
\text{the function $t\mapsto t\Lambda'(t)$ does not decrease,}
\label{y29}
\end{equation}
and \eqref{y9} holds, then $S$ is cyclic in $\mathcal{B}_\Lambda ^\infty$. Since the proof relied on the quasianalyticity property 
of an auxiliary class of functions, this convexity type condition \eqref{y29} was indispensable here.

In the first part of our paper we prove the following results:

\begin{theo}
Let $\Lambda$ be a positive non-increasing continuous function on $(0,1]$.
Then $S(z)=e^{-\frac{1+z}{1-z}}$ is cyclic in $\mathcal{B}_\Lambda^{\infty,0}$ if and only if $\Lambda$ satisfies \eqref{y9}. 
\label{tem}
\end{theo}

\begin{theo}
Let $1\le p<\infty$ and let $\Lambda$ be a positive non-increasing continuous function on $(0,1]$. Then $S(z)=e^{-\frac{1+z}{1-z}}$ is cyclic in $\mathcal{B}_\Lambda^p$ if and only if $\Lambda$ satisfies \eqref{y9}.
\label{tembis}
\end{theo}

For every $\Lambda$, $1\le p<\infty$ we have
$$
\mathcal{B}_\Lambda^{\infty,0}=\mathcal{B}_{\widetilde{\Lambda}}^{\infty,0},\qquad
\mathcal{B}_\Lambda^p=\mathcal{B}_{\widetilde{\Lambda}}^p,
$$
where 
$$
\widetilde{\Lambda}(t)=\inf_{s\le t}\Lambda(s);
$$
the function $\widetilde{\Lambda}$ is non-decreasing.
Thus, we are able to get rid of any regularity condition  
on $\Lambda$. That is possible because we use the so-called resolvent transform method exposed, for example, in \cite{YD,BH}. This technique was introduced by Carleman and Gelfand; 
it was later rediscovered and used upon by Domar 
to study closed ideals in Banach algebras. 

An open question here is whether the resolvent transform method could help in improving 
the result of Nikolski \cite[Section 2.6]{NN} concerning cyclicity of zero-free functions of restricted growth (say $f\in \mathcal{B}_{\Lambda_1}^\infty$, $f$ is zero-free implies 
that $f$ is cyclic in $\mathcal{B}_{\Lambda}^{\infty,0}$ under some conditions on the pair 
$(\Lambda_1,\Lambda)$).

One more remark is that the Hadamard three--circles theorem (see, for instance, \cite[Theorem 3.5, Chapter 8]{LA}) shows that we can replace $\Lambda$ by its log-convex minorant $\widetilde{\Lambda}$; the function $t\mapsto (1-t)\Lambda'(t)$ does not decrease. Of course, this condition is much weaker than \eqref{y29}.

Let us also mention here that the convergence of the integral in \eqref{y9} appears 
in the work of Hayman and Korenblum \cite{HK1,HK2} as necessary and sufficient for the 
existence of some Nevanlinna-type representation and for validity of the Blaschke 
conditions for zeros on radii for functions in $\mathcal{B}_{\Lambda}^\infty$. On the other hand, the Blaschke type conditions in spaces $\mathcal{B}_{\Lambda,E}^{\infty}$ 
were recently discussed in \cite{BGK} and \cite{FG}, see also 
the references therein.

In 1986 Gevorkyan and Shamoyan \cite{GS} obtained 
a necessary and sufficient condition, 
\begin{equation}
\int_0\Lambda(t)\,dt=\infty,
\label{gs}
\end{equation} 
for cyclicity of $S$ in 
$\mathcal{B}_{\Lambda,\{1\}}^{\infty,0}$ under some regularity conditions on $\Lambda$.
Recently, El-Fallah, Kellay, and Seip \cite{EKS} improved the results of Beurling and Nikolski 
for cyclicity of bounded zero-free functions in $H_w^2$ spaces. Furthermore, improving on the result by Gevorkyan--Shamoyan they obtained that \eqref{gs} is necessary and sufficient for cyclicity of $S$ in 
$\mathcal{B}_{\Lambda,\{1\}}^{\infty,0}$, the only regularity condition being that $\Lambda$ is decreasing. Their method of proof (applying the Corona theorem) 
is quite different from what we use here. 

It is now a natural question to describe 
$\Lambda$ such that $S$ is cyclic in $\mathcal{B}_{\Lambda,E}^{\infty,0}$, $\mathcal{B}_{\Lambda,E}^2$ in terms of the behavior of $E$ near the point $1$. Our method 
applies directly in the case where $E$ is a closed arc containing $1$ (see Remark~\ref{r6} at the end of Section~\ref{s4}). 

In the second part of our paper, for sufficiently regular $\Lambda$ 
we get necessary and sufficient conditions in terms of $E$
(Theorems~\ref{teo0} and \ref{teo1}). 
More precisely, for $\Lambda$ defined by 
$\Lambda(t)=\frac1{tw(t)^2}$ with sufficiently regular $w$, 
the function $S$ is cyclic in $\mathcal{B}_{\Lambda,E}^{\infty,0}$ if and only if one of the following three quantities is infinite: 
$$
\int_{e^{it}\in E}\frac{dt}{|t|w(|t|)},\quad  
\int_{0}\frac{dt}{|t|w^2(|t|)},\quad 
\sum_{n}\frac1{w(b_n)^2}
\log\Bigl[1+\Bigl(1-\frac{a_n}{b_n}\Bigr)w(b_n)\Bigr],
$$ 
where the sums runs by all the complementary arcs to $E$: 
$(e^{ia_n},e^{ib_n})$, $(e^{-ib_n},e^{-ia_n})$, $0<a_n<b_n$.
As applications, we get 
a result (Theorem~\ref{teo2}) interpolating between the theorems of Nikolski and Gevorkyan--Shamoyan, 
and a result when $E$ is the Cantor ternary set (Theorem~\ref{teo3}).

We prove our Theorems~\ref{tem} and \ref{tembis} in Section~\ref{s4}. Before that, we establish 
a Phragm\'en--Lindel\"of type theorem in Section~\ref{s2} and 
prove some growth estimates on auxiliary outer functions in 
Section~\ref{s3}. Finally, for general $E$ we construct an auxiliary domain in Section~\ref{s5} and 
obtain cyclicity results 
in Section~\ref{s6}.

\section{Generalized Phragm\'en--Lindel\"of Principle}
\label{s2}

We start with the Ahlfors--Carleman estimate of the harmonic measure (see, for example, \cite[IX.E1]{K}). 
Let $G$ be a simply connected domain such that  
$\infty\in\partial G$. Fix $z_0\in G$. 
Given $\rho>0$ we denote by $G_\rho$ the connected component 
of the intersection of the disc $\rho\mathbb D=\{z:|z|<\rho\}$ 
and $G$ containing $z_0$, 
$S_{\rho}$ is an arc on $\partial(\rho\mathbb D)\cap G$ 
separating $z_0$ from one of the unbounded components of 
$G\setminus \overline{\rho\mathbb D}$, $s(\rho)$ is the length of $S_\rho$. Then
$$
\omega(z_0,S_\rho,G_\rho)\le C\exp\Bigl(-\pi\int_0^\rho\frac{dr}{s(r)}\Bigr)
$$
for an absolute constant; here $\omega(z_0,\cdot,\Omega)$
is the harmonic measure with respect to $z\in\Omega$ on the boundary of $\Omega$. 

\begin{coro}\label{ty}
Suppose that $G$ is as above, a function $f$ is analytic in $G$, continuous up to 
$\partial G\setminus\{\infty\}$ and satisfies the conditions
\begin{gather*}
|f(z)|\le 1, \qquad z\in\partial G,\\   
\liminf_{\rho\to\infty}\frac{\log M(\rho)}{\sigma(\rho)}=0,
\end{gather*}
where
$$
M(\rho)=\max_{z\in S_\rho}|f(z)|,\qquad 
\sigma(\rho)=\exp\Big\{\pi\int_1^\rho\frac{dr}{s(r)}\Big\}.
$$
Then $|f(z)|\le 1$, $z\in G$.
\end{coro}

Given an increasing differentiable function $\phi:[0,+\infty[\to\RR_{+}$ such that
\begin{equation} 
\lim_{x\to\infty}\frac{\phi(x)}{x}=\infty,
\label{y3}
\end{equation} 
we consider the domain $G_\phi$ defined by:
$$
G_\phi:=\big\{x+iy\in\CC : |y|\le \phi(x),\,x\ge 0 \big\}.
$$

\begin{proposition}
Let $f$ be a function analytic on $G_\phi$ and continuous up to 
$\partial G_\phi\setminus\{\infty\}$ such that for some $c>0$ we have
\begin{gather*}
|f(z)|\le e^{c|z|},\qquad z\in G_{\phi},\\
|f(\xi)|\le 1,  \qquad \xi\in \partial G_\phi\setminus\{\infty\}.
\end{gather*}
If 
\begin{equation}
\int^\infty\frac{x\phi'(x)\,dx}{\phi(x)^2}=+\infty,
\label{y2}
\end{equation} 
then $|f(z)|\le 1$, $z\in G_\phi$.
\label{y10}
\end{proposition}

\begin{proof} In the notations of Corollary~\ref{ty}, we have 
$\log M(\rho)\le c\rho$.

A simple geometric argument shows that if 
\begin{equation}
r^2=x^2+\phi(x)^2,
\label{y1}
\end{equation} 
then
$$
s(r)=r\Bigl(\pi-2\arctan\frac{x}{\phi(x)}\Bigr).
$$
Therefore, if $x=x(r)$ is defined by \eqref{y1}, then 
$$
\frac{\pi}{s(r)}-\frac{1}{r}\ge \frac{x}{2r\phi(x)}\ge \frac{x}{3\phi(x)^2}
$$
for large $r$ (we use \eqref{y3}). Moreover, 
$$
dr=\frac{x+\phi(x)\phi'(x)}{r}\,dx\ge \frac{\phi'(x)\,dx}{2}
$$
for large $r$.

Thus, for some $c>0$ we have 
$$
\sigma(\rho)\ge c\rho\cdot\exp\Bigl[\frac16\int_{x(1)}^{x(\rho)}
\frac{x\phi'(x)\,dx}{\phi(x)^2}\Bigr].
$$
Therefore, \eqref{y2} implies that
$$
\lim_{\rho\to \infty}\frac{\log M(\rho)}{\sigma(\rho)}=0,
$$
and it remains to apply Corollary~\ref{ty}.
\end{proof}

By the standard maximum principle, it suffices to require in Proposition~\ref{y10} that 
$\phi(x)$ is increasing for large $x$. 
\smallskip

Given a positive decreasing differentiable function $\Lambda$ 
on $(0,1]$ such that $\Lambda(0^{+})=+\infty$ and 
\begin{equation}
t\Lambda(t)=o(1), \qquad t\to 0,
\label{y8}
\end{equation} 
we consider the domain  
$$
\Omega_\Lambda=\Big\{w\in\DD : \frac{1-|w|^{2}}{|1-w|^{2}}\ge\Lambda(1-|w|^2)\Big\}.
$$
Let $F(w)=\frac{1+w}{1-w}$ be a conformal map of the unit disc onto 
the right half  plane, and let $x+iy$ be a point on the boundary of $F(\Omega_\Lambda)$, $y\ge 0$.  
Then 
$$
x=\Lambda\Bigl(\frac{4x}{(x+1)^2+y^2}\Bigr).
$$
By monotonicity of $\Lambda$, $y$ is determined uniquely by $x$, $y=y(x)$, and we have 
$F(\Omega_\Lambda)=G_y$.
Furthermore, we obtain that
$$
\frac{4x^2}{(x+1)^2+y^2}=\frac{4x}{(x+1)^2+y^2}\cdot \Lambda\Bigl(\frac{4x}{(x+1)^2+y^2}\Bigr)=o(1),\qquad x\to\infty.
$$
Hence $x=o(y)$, $x\to\infty$. 

Next, for sufficiently small positive $t$ we have 
$$
\frac{4\Lambda(t)}{(\Lambda(t)+1)^2+y(\Lambda(t))^2}=t,
$$
and, hence, $y$ is differentiable and 
$$
2\frac{t\Lambda'(t)-\Lambda(t)}{t^2}=\Lambda'(t)\bigl(\Lambda(t)+1+y(\Lambda(t))y'(\Lambda(t))\bigr).
$$
Since $t\Lambda(t)=o(1)$ and $y(\Lambda(t))=(2+o(1))\sqrt{\Lambda(t)/t}$, $t\to 0$, 
we obtain
\begin{equation}
\Lambda'(t)y'(\Lambda(t))=(1+o(1))\frac{t\Lambda'(t)-\Lambda(t)}{t\sqrt{t\Lambda(t)}}, \qquad t\to 0.
\label{y5}
\end{equation}
In particular, the function $y(x)$ increases for large $x$.

Finally, we estimate the integral
$$
I=\int^\infty\frac{xy'(x)}{y(x)^2}\,dx.
$$
By \eqref{y5} we have
$$
I\ge c+\int_0\frac{t}{5}\cdot\frac{|t\Lambda'(t)-\Lambda(t)|}{t\sqrt{t\Lambda(t)}}
\,dt
=c+\frac{1}{5}\int_{0}\sqrt{\frac{\Lambda(t)}{t}}\,dt+\frac{1}{5}\int_{0}|\Lambda'(t)|
\sqrt{\frac{t}{\Lambda(t)}}\,dt.
$$
Integrating by parts and using \eqref{y8}, one can easily verify that the integrals 
$\int_0\sqrt{\Lambda(t)/t}\,dt$ and $\int_0|\Lambda'(t)|\sqrt{t/\Lambda(t)}\,dt$
converge simultaneously.
Thus, $I=\infty$ if and only if \eqref{y9} holds.

\begin{coro}\label{coro1}
Let $f$ be analytic on the domain $\Omega_{\Lambda}$ and continuous up to 
$\partial\Omega_\Lambda\setminus\{1\}$, where $\Lambda$ 
is a positive decreasing differentiable function  
on $(0,1]$ satisfying $\Lambda(0^{+})=+\infty$ and \eqref{y8}.
If for some $c>0$ we have
\begin{enumerate}
\item[(a)] $|f(w)|\le e^{c\frac{1-|w|^{2}}{|1-w|^{2}}}$, $w\in\Omega_\Lambda$,
\item[(b)] $|f(\xi)|\le 1$, $\xi\in\partial\Omega_{\Lambda}\setminus\{1\}$,
\end{enumerate}
and if 
$$
\int_0\sqrt{\frac{\Lambda(t)}{t}}\,dt=\infty,
$$
then $|f(z)|\le 1$, $z\in\Omega_\Lambda$.
\end{coro}

\section{Auxiliary estimates}
\label{s3}

Given $\lambda\in\DD$, we define the Privalov shadow $I_\lambda$,
$$
I_\lambda=\bigl\{e^{i\theta}:|\arg(\lambda e^{-i\theta})|<\frac{1-|\lambda|}2\bigr\} 
$$
(see Figure \ref{fig:privalov}), which is an arc of the circle $\TT$ centered at 
$\frac{\lambda}{|\lambda|}$ (the radial projection of $\lambda$ onto the unit circle $\TT$) of length $|I_\lambda|=1-|\lambda|$. 
Furthermore, we consider an auxiliary function
$$
f_\lambda(z)=\exp\Big\{c_{\lambda}\frac{1-|\lambda|^{2}}{|1-\lambda|^{2}}\int_{I_\lambda}\frac{e^{i\theta}+z}{e^{i\theta}-z}\,d\theta\Big\},
$$
where
$$
c_{\lambda}^{-1}=\int_{I_\lambda}\frac{1-|\lambda|^{2}}{|e^{i\theta}-\lambda|^{2}}
\,d\theta.
$$
\begin{center}
\begin{figure}[h]
\includegraphics[width=0.4\linewidth]{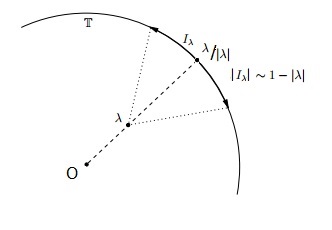}
\caption{\label{fig:privalov}Privalov shadow.}
\end{figure}
\end{center}
Then
$$
|f_\lambda(\lambda)S(\lambda)|=1,
$$
where $S(z)=\exp[-(1+z)/(1-z)]$.
Next, a geometric argument shows that 
$$
c_\lambda^{-1}\ge \frac{1-|\lambda|^2}{5/4(1-|\lambda|)^2}\cdot(1-|\lambda|)\ge 
\frac45,
$$
and we have 
$$
\sup_\DD|f_\lambda|\le \exp\Big(\frac{5\pi}2\cdot\frac{1-|\lambda|^2}{|1-\lambda|^2}\Big).
$$

Given $a>0$, $A>1$, we consider the domain 
$\Gamma_{\Lambda,\T}(a,A)$ defined by 
$$
\Gamma_{\Lambda,\T}(a,A)=\Big\{\lambda\in\DD:\frac{1-|\lambda|^2}{|1-\lambda|^2}\le 
a\Lambda(A(1-|\lambda|))\Big\}.
$$

Now for some $a,A$ we establish the following estimate: 

\begin{lem}\label{y11}
\begin{equation}
\sup_{\lambda\in\Gamma_{\Lambda,\T}(a,A)}\|f_\lambda S\|_{\Lambda}<\infty.
\label{y11t}
\end{equation}
\end{lem}

\begin{proof} We set
$$
H_\lambda(z)=c_\lambda\,\frac{1-|\lambda|^{2}}{|1-\lambda|^{2}}\int_{I_\lambda}\frac{1-|z|^2}{|e^{i\theta}-z|^2}\,d\theta-\frac{1-|z|^2}{|1-z|^2}-\Lambda(1-|z|),
$$
and obtain
$$
|f_\lambda(z)S(z)|e^{-\Lambda(1-|z|)}=e^{H_\lambda(z)}.
$$
Thus, it remains to verify that
$$
\sup_{\lambda\in\Gamma_{\Lambda,\T}(a,A),\, z\in\DD}H_\lambda(z)<\infty.
$$ 
\smallskip

{\bf Case 1:} If $A(1-|\lambda|)\ge 1-|z|$, then  
\begin{multline*}
H_\lambda(z)\le 2\pi c_\lambda\,\frac{1-|\lambda|^2}{|1-\lambda|^2}-\Lambda(1-|z|)\\ 
\le \frac{5\pi a}2\,\Lambda(A(1-|\lambda|))-\Lambda(A(1-|\lambda|))=\Bigl(\frac{5\pi a}2-1\Bigr)\Lambda(A(1-|\lambda|)).
\end{multline*}
Therefore, for $a\le\frac{2}{5\pi}$ we obtain that $H_\lambda(z)\le 0$. 
\smallskip

{\bf Case 2:} If $|1-z|\ge 6|1-\lambda|$, then 
for every $e^{i\theta}\in I(\lambda)$ we have 
\begin{multline*}
|z-e^{i\theta}|\ge|1-z|-|1-\lambda|-|\lambda-e^{i\theta}|
\ge |1-z|-|1-\lambda|-2(1-|\lambda|)
\\ \ge|1-z|-3|1-\lambda|
\ge \frac12|1-z|.
\end{multline*}
Therefore,
\begin{multline*}
H_\lambda(z)\le 5\frac{1-|\lambda|^2}{|1-\lambda|^2}\frac{1-|z|^2}{|1-z|^2}(1-|\lambda|)-\frac{1-|z|^2}{|1-z|^2}
\le \Big[5 a(1-|\lambda|)\Lambda(A(1-|\lambda|))-1\Big]\frac{1-|z|^2}{|1-z|^2}.
\end{multline*}
Since $t\Lambda(t)=o(1)$, $t\to 0,$ we obtain that $H_\lambda(z)$ is uniformly bounded.
\smallskip

{\bf Case 3:} If $A(1-|\lambda|)<1-|z|$ and $|1-z|<6|1-\lambda|$, then
$$
\frac{1-|z|^{2}}{|1-z|^{2}}\geq\frac{A}{72}\cdot \frac{1-|\lambda|^2}{|1-\lambda|^2},
$$
and hence,
$$
H_\lambda(z)\le \frac{180\pi}{A} \cdot\frac{1-|z|^2}{|1-z|^2}-\frac{1-|z|^2}{|1-z|^2}\le 0,
$$  
for $A\ge 1000$.
\end{proof}

Now, let $E$ be an arbitrary compact subset of $\T$. 
For some $a>0$, $A>1$, we consider the domain 
$\Gamma_{\Lambda,E}(a,A)$ defined by 
$$
\Gamma_{\Lambda,E}(a,A)=\Big\{\lambda\in\DD:\frac{1-|\lambda|^2}{|1-\lambda|^2}\le 
a\Lambda(A\dist(\lambda,E))\Big\},
$$
and obtain the estimate

\begin{lem}
$$
\sup_{\lambda\in\Gamma_{\Lambda,E}(a,A)}\|f_\lambda S\|_{\Lambda}<\infty.
$$
\label{y11bis}
\end{lem}

\begin{proof} We just need to verify that  
$$
\sup_{\lambda\in\Gamma_{\Lambda,E}(a,A),\, z\in\DD}H_\lambda(z)<\infty.
$$ 

In the cases (1) $A\dist(\lambda,E)\ge \dist(z,E)$ and (2) $|1-z|\ge 6|1-\lambda|$
we use the same argument as in the proof of Lemma~\ref{y11}. In the case (3) we have
$A\dist(\lambda,E)<\dist(z,E)$ and $|1-z|<6|1-\lambda|$. Take 
$e^{i\eta}\in E$ such that $A|\lambda-e^{i\eta}|\le |z-e^{i\eta}|$ and use that for 
$e^{i\theta}\in I_{\lambda}$ we have 
$$
1-|\lambda|\le |\lambda-e^{i\theta}|\le 2 |\lambda-e^{i\eta}|.
$$
Then,
$$
|z-e^{i\theta}|\ge |z-e^{i\eta}|-|\lambda-e^{i\theta}|-|\lambda-e^{i\eta}|\ge 
(A-3)|\lambda-e^{i\eta}|\ge \frac{A-3}{2}(1-|\lambda|),\qquad e^{i\theta}\in E.
$$
Therefore,
$$
H_\lambda(z)\le \frac{5}{4}\frac{1-|\lambda|^2}{|1-\lambda|^2}\frac{4}{(A-3)^2}\frac{1-|z|^2}{|1-z|^2}(1-|\lambda|)-\frac{1-|z|^2}{|1-z|^2}\le
\Bigl(\frac{360}{(A-3)^2}-1\Bigr)\frac{1-|z|^2}{|1-z|^2}\le 0
$$  
for $A\ge 100$.
\end{proof}

\section{Proofs of Theorems~\ref{tem} and \ref{tembis}}
\label{s4}

\begin{proof}[Proof of Theorem~\ref{tem}] Suppose that $\Lambda$ is non-increasing 
and satisfies \eqref{y9}. Then $\Lambda(0^+)=\infty$, and slightly decreasing, if necessary, $\Lambda$, we can assume that $\Lambda$ is decreasing.

If $\int_0 \Lambda(t)\,dt=\infty$, then the result follows from 
\cite[Theorem 2]{EKS}. Therefore, from now on we assume that 
$\int_0\Lambda(t)\,dt<\infty$. Since $\Lambda$ decreases, we have
$$
\Lambda(t)\le \frac{1}{t}\int_0^t\Lambda(s)\,ds= o\Big(\frac 1t\Big),\qquad t\to 0,
$$
and, hence, $t\Lambda(t)=o(1)$, $t\to 0$. Finally, smoothing, if necessary, $\Lambda$, 
we can assume that $\Lambda$ is differentiable.

We denote by $\pi$ the canonical projection of $\mathcal{B}_\Lambda^{\infty,0}$ onto $\mathcal{B}_\Lambda^{\infty,0}/[S]$, and by $\alpha :z\mapsto z$ the identity map. 

Suppose that $1\not\in[S]$. Next, we estimate $\|(\lambda-\pi(\alpha))^{-1}\pi(1)\|$. Given $\lambda\in\DD$ and an analytic function $f$, 
we define the following function:
$$
L_{\lambda}(f)(z)=\left\{
\begin{array}{cll}
\dfrac{f(z)-f(\lambda)}{z-\lambda}   &if& z\neq \lambda \\ \\
f'(z)   &if& z=\lambda\,.
\end{array}
\right.
$$
For bounded functions $f$ we have
\begin{equation} 
\label{y14}
f(\lambda)S(\lambda)\pi(1)=(\lambda-\pi(\alpha))\pi(L_\lambda(fS)).
\end{equation}
In particular,
$$
S(\lambda)\pi(1)=(\lambda-\pi(\alpha))\pi(L_\lambda(S)),\qquad \lambda\in \mathbb C\setminus\{1\},
$$
and hence, the function $\lambda\mapsto(\lambda-\pi(\alpha))^{-1}\pi(1)$ is 
well defined and analytic in 
$\mathbb C\setminus\{1\}$.

Next we use that the function $Q(z)=\log|f_\lambda(z)S(z)|$ is the Poisson integral 
$$
Q(z)=\mathcal P(z,\theta)*d\mu(\theta)
$$
of a finite (signed) measure 
$$
d\mu(\theta)=2\pi c_\lambda\frac{1-|\lambda|^{2}}{|1-\lambda|^{2}}
\chi_{I_\lambda}(\theta)\,d\theta-\delta_1
$$
with mass 
$$
|\mu|(\T)=2\pi c_\lambda\frac{(1-|\lambda|^2)(1-|\lambda|)}{|1-\lambda|^{2}}+1
$$
uniformly bounded in $\lambda$. Since 
$$
\sup_{z\in\DD,\,\theta\in\T}(1-|z|)^2\|\nabla P(z,\theta)\|<\infty,
$$
for an absolute constant $c$ we obtain that 
$$
\|\nabla Q(z)\|\le \frac{c}{(1-|z|)^2}.
$$ 
Since $Q(\lambda)=0$, 
we have $|f_\lambda(z)S(z)|\le c_1$ at the boundary of the disc 
$D_\lambda=\{z:|z-\lambda|<(1-|\lambda|)^2/2\}$, and hence, by the maximum principle, 
$$
\Big|\frac{f_\lambda(z)S(z)-f_\lambda(\lambda)S(\lambda)}{z-\lambda}\Big|\le
\frac{2(c_1+1)}{(1-|\lambda|)^2},\qquad z\in D_\lambda.
$$

Therefore,
\begin{multline*}
\|L_\lambda(f_\lambda S)\|_\Lambda=\sup_{|z|<1}\Big|\frac{f_\lambda(z)S(z)-f_\lambda(\lambda)S(\lambda)}{z-\lambda}\Big|e^{-\Lambda(1-|z|)}\\
\le\sup_{|z-\lambda|\le(1-|\lambda|)^2/2}\Big|\frac{f_{\lambda}(z)S(z)-f_{\lambda}(\lambda)S(\lambda)}{z-\lambda}\Big|e^{-\Lambda(1-|z|)}\\
+\sup_{|z-\lambda|>(1-|\lambda|)^2/2}\Big|\frac{f_\lambda(z)S(z)-f_\lambda(\lambda)S(\lambda)}{z-\lambda}\Big|e^{-\Lambda(1-|z|)}\\
\le\frac{2}{(1-|\lambda|)^2}\big(\|f_{\lambda}S\|_{\Lambda}+c_1+2\big).
\end{multline*}

By \eqref{y14}, for every $\lambda\in\DD$ we have
\begin{equation}
\|(\lambda-\pi(\alpha))^{-1}\pi(1)\|_\Lambda\le\frac{1}{|f_\lambda(\lambda)S(\lambda)|} \|L_\lambda(f_\lambda S)\|_\Lambda\le \frac{2}{(1-|\lambda|)^2}\big(\|f_{\lambda}S\|_{\Lambda}+c_1+2\big). 
\label{y16}
\end{equation}

In the same way, by \eqref{y14}, for every $\lambda\in\DD$ we have
\begin{equation} 
\label{eq1}
\|(\lambda-\pi(\alpha))^{-1}\pi(1)\|_\Lambda\le \frac1{S(\lambda)}
\|\pi(L_\lambda(S))\|_\Lambda \le
\frac 2{1-|\lambda|}\exp\frac{1-|\lambda|^2}{|1-\lambda|^2}.
\end{equation}
Furthermore, for $|\lambda|>1$ we have
\begin{equation} 
\label{y19}
\|(\lambda-\pi(\alpha))^{-1}\pi(1)\|_{\Lambda}\le \frac1{|S(\lambda)|}\|\pi(L_\lambda(S))\|_\Lambda \le \frac2{|\lambda|-1}.
\end{equation}

Let us fix $a$ and $A$ such that \eqref{y11t} holds.
For $\lambda\in\partial\Gamma_\Lambda(a,A)\setminus\TT$ we have $1-|\lambda|^{2}= a\Lambda(A(1-|\lambda|))|1-\lambda|^{2}$, and hence, 
\begin{equation} 
\label{eq2}
\frac{1}{1-|\lambda|}\leq\frac{c}{|1-\lambda|^2},\qquad\lambda\in\partial\Gamma_\Lambda(a,A)\setminus\TT.
\end{equation}
By Lemma~\ref{y11} and \eqref{y16}, we have 
\begin{equation} 
\label{y18}
\sup_{\lambda\in\partial\Gamma_\Lambda(a,A)}\|(\lambda-1)^4(\lambda-\pi(\alpha))^{-1}\pi(1)\|<
\infty.
\end{equation}
Applying \eqref{eq1} and Corollary~\ref{coro1} (to $\Lambda_1$ defined 
by $\Lambda_1(1-r^2)=a\Lambda(A(1-r))$), we obtain that 
the function $(\lambda-1)^4(\lambda-\pi(\alpha))^{-1}\pi(1)$ is bounded on $\DD\setminus\Gamma_\Lambda(a,A)$.

By \eqref{eq2}, for some $c>0$ we have
$$
\dist(\zeta,\partial\Gamma_\Lambda(a,A)\setminus\TT)\ge c|1-\zeta|^2,\qquad \zeta\in\T.
$$
Applying Levinson's log-log theorem (see, for example, \cite[VII D7]{K}) or, rather, its polynomial growth version to the function $(\lambda-\pi(\alpha))^{-1}\pi(1)$ and using estimates \eqref{y18} and \eqref{y19} 
we obtain that 
\begin{equation} 
\label{y21}
\|(\lambda-\pi(\alpha))^{-1}\pi(1)\|\leq \frac{C}{|\lambda-1|^4}, \qquad  \lambda\in
2\DD\setminus(\{1\}\cup\DD\setminus\Gamma_\Lambda(a,A)).
\end{equation}
By \eqref{y18} and \eqref{y21}, the function $\lambda\mapsto (\lambda-1)^4(\lambda-\pi(\alpha))^{-1}\pi(1)$ is bounded on $2\DD\setminus\{1\}$.

Now, pick an arbitrary functional $\phi\bot [S]$, and consider the function
$$
\Phi(\lambda)=(\lambda-1)\langle(\lambda-\pi(\alpha))^{-1}\pi(1),\phi\rangle.
$$
The function $\Phi$ is analytic on $\CC\setminus\{1\}$ and has a pole of order at most $3$ at the point $1$. By \eqref{y19}, we have 
$$
\limsup_{\varepsilon\to 0^+}|\Phi(1+\varepsilon)|<\infty,
$$
and hence, $\Phi$ is an entire function. Again by \eqref{y19}, $\Phi$ is bounded 
and hence is a constant function.
Furthermore,
\begin{multline*}
\Phi(\lambda)=\langle(\lambda-\pi(\alpha))^{-1}\pi(\lambda-1),\phi\rangle=
\langle(\lambda-\pi(\alpha))^{-1}\pi(\lambda-\alpha),\phi\rangle+
\langle(\lambda-\pi(\alpha))^{-1}\pi(\alpha-1),\phi\rangle\\=
\langle(\pi(1),\phi\rangle+
\langle(\lambda-\pi(\alpha))^{-1}\pi(\alpha-1),\phi\rangle.
\end{multline*}
Arguing as in the proof of \eqref{y19}, we see that
$$
\lim_{|\lambda|\to\infty}(\lambda-\pi(\alpha))^{-1}\pi(\alpha-1)=0,
$$
and hence, $\langle\alpha-1,\phi\rangle=0$. Since $\phi$ is an arbitrary functional 
vanishing on $[S]$, we get $\alpha-1\in[S]$, $\alpha^n-1\in[S]$, $n\ge 1$, and hence, $1\in[S]$. This contradiction shows that $[S]=\mathcal{B}_\Lambda^{\infty,0}$.

In the opposite direction, suppose that $S$ is cyclic in 
$\mathcal{B}_\Lambda^{\infty,0}$. We replace $\Lambda$ by $\Lambda^*$, 
$$
\Lambda^*(t)=\Lambda(t)+\log\frac1t,
$$
and remark that $S$ is also cyclic in $\mathcal{B}_{\Lambda^*}^{\infty,0}$. Now we can apply the result by Nikolski \cite[Theorem 1a, Section 2.6]{NN} to conclude that 
$\Lambda^*$ satisfies \eqref{y9}.
\end{proof}

\begin{proof}[Proof of Theorem~\ref{tembis}]
It suffices to remark that 
$$
\mathcal{B}_\Lambda^{\infty,0}\subset\mathcal{B}_\Lambda^p\subset
\mathcal{B}_{\Lambda^*}^{\infty,0},
$$
where 
$$
\Lambda^*(t)=\Lambda(t/2)+\frac2p \log\frac1t,
$$ 
and to apply Theorem~\ref{tem}. We use here the subharmonicity of $|f|^p$ and the fact that $\Lambda$ and $\Lambda^*$ satisfy \eqref{y9} simultaneously.
\end{proof}

\begin{remq} 
\label{r6}
The same method works for the spaces $\mathcal{B}_{\Lambda,E}^p$, where $E$ is a closed arc of the unit circle. More precisely we have the following result:

Let $E$ be a non-trivial closed arc of the unit circle such $1\in E$, let $1\le p<\infty$, and let $\Lambda$ be a positive non-increasing continuous function on $(0,1]$. Then $S(z)=e^{-\frac{1+z}{1-z}}$ is cyclic in $\mathcal{B}_{\Lambda , E}^p$ if and only if $\Lambda$ satisfies \eqref{y9}.
\end{remq}

\noindent
\begin{proof}[Sketch of the proof:]
For the necessity part it suffices to use that $\mathcal{B}_{\Lambda , E}^p \subset \mathcal{B}_{\Lambda }^p$. In the opposite direction, suppose that  $\Lambda$ satisfies \eqref{y9} and let $E= {\overline{(e^{ia},e^{ib})}}$. If $1\in (e^{ia},e^{ib})$ then the same proof works. In the case $e^{ia}=1$ a slight modification is needed:
we replace $\Omega_\Lambda$ by
\begin{multline*}
\widetilde{\Omega}_\Lambda=\Bigl\{w\in\DD:\Im \,w\ge 0,\,
\frac{1-|w|^{2}}{|1-w|^{2}}\ge a\Lambda(A(1-|w|^2))\Bigr\}\\ \cup
\Bigl\{w\in\DD:\Im\,w\le 0,\,
\frac{1-|w|^{2}}{|1-w|^{2}}\ge a_1\Lambda(A_1|1-w|)\Bigr\}
\end{multline*}
for some $a,a_1,A,A_1$.
\end{proof} 

\section{An auxiliary domain for general $E$}
\label{s5}

Let $E$ be a compact subset of $\mathbb T$, $1\in E$. Given a positive decreasing $C^1$ smooth function 
$\Lambda$ on $(0,1)$ such that $\Lambda(1)<1/10$, and 
\begin{equation}
\lim_{t\to0}\Lambda(t)=\infty,\qquad 
\lim_{t\to0}t\Lambda(t)=0,\qquad 
t|\Lambda'(t)|=O(\Lambda(t)),\quad t\to 0,
\label{cond1}
\end{equation}
we consider the domain 
$$
\Omega_{\Lambda,E}=\{(1-s)e^{i\theta}:s\ge \theta^2\Lambda(s+\dist(e^{i\theta},E)),\,\theta\in(-\pi,\pi]\}.
$$
Clearly, $\Omega_{\Lambda,E}$ is star-shaped with respect to the origin, 
$$
\partial\Omega_{\Lambda,E}=\{(1-\gamma(\theta))e^{i\theta},\,\theta\in(-\pi,\pi]\}
$$
with 
\begin{equation}
\gamma(\theta)=\theta^2\Lambda(\gamma(\theta)+\dist(e^{i\theta},E)).
\label{ga1}
\end{equation}
Next, 
\begin{equation}
\inf_{\theta\in(-\pi,\pi]\setminus\{0\}}\frac{\gamma(\theta)}{\theta^2}>0, 
\label{dop1}
\end{equation}
$\lim_{\theta\to 0}\gamma(\theta)=0$, and
$$
\frac{\gamma(\theta)}{\theta}=
[\gamma(\theta)\Lambda(\gamma(\theta)+\dist(e^{i\theta},E))]^{1/2}
=o(1),\qquad \theta\to 0.
$$
Furthermore, the derivative $h(\theta)$ of $\dist(e^{i\theta},E)$ 
is equal to $\pm 1$ for a.e. $e^{i\theta}$ on $\mathbb T\setminus E$ and to $0$ for 
a.e. $e^{i\theta}$ on $E$. Therefore,
$$
\gamma'(\theta)=2\theta\Lambda(\gamma(\theta)+\dist(e^{i\theta},E))-\theta^2|\Lambda'(\gamma(\theta)+\dist(e^{i\theta},E))|(\gamma'(\theta)+h(\theta))
$$
for a.e. $e^{i\theta}\in\mathbb T$, 
and hence, by \eqref{cond1}, 
\begin{equation}
|\gamma'(\theta)|=O(1),\qquad \text{a.e.\ }\theta\to 0.
\label{condO}
\end{equation}
Now we set
$$
(1-\gamma(\theta))e^{i\theta}=1-\frac{e^{i\phi}}{R},
$$
with $\phi\in[-\pi/2,\pi/2]$.
Then
\begin{gather*}
R\asymp\frac1\theta,\qquad \frac\pi2-|\phi|\asymp \frac{\gamma(\theta)}{\theta},\\ \frac{dR}{d\theta}=(1+o(1))R^2,\qquad \theta\to 0.
\end{gather*}

By analogy with Proposition~\ref{y10} and Corollary~\ref{coro1} we have 
 
\begin{proposition}
Given a continuous function $\phi:\mathbb R_+\to(0,\pi/2)$ 
let 
$$
G_\phi:=\big\{Re^{i\theta}:|\theta|< \frac\pi2-\phi(R),\,R>0 \big\}.
$$
Let $f$ be a function analytic on $G_\phi$ and continuous up to 
$\partial G_\phi\setminus\{\infty\}$ such that for some $c>0$ we have
\begin{gather*}
|f(z)|\le e^{c|z|},\qquad z\in G_{\phi},\\
|f(\xi)|\le 1,  \qquad \xi\in \partial G_\phi\setminus\{\infty\}.
\end{gather*}
If 
\begin{equation}
\int^\infty\frac{\phi(R)\,dR}{R}=+\infty,
\label{y2nn}
\end{equation} 
then $|f(z)|\le 1$, $z\in G_\phi$.
\label{y10nn}
\end{proposition}

\begin{coro}
Let $f$ be analytic on the domain $\Omega_{\Lambda,E}$ and continuous up to 
$\partial\Omega_{\Lambda,E}\setminus\{1\}$, where $\Lambda$ 
is a positive decreasing differentiable function  
on $(0,1]$ satisfying \eqref{cond1}.
If for some $c>0$ we have
\begin{enumerate}
\item[(a)] $|f(w)|\le e^{c\frac{1-|w|^{2}}{|1-w|^{2}}}$, $w\in\Omega_{\Lambda,E}$,
\item[(b)] $|f(\xi)|\le 1$, $\xi\in\partial\Omega_{\Lambda,E}\setminus\{1\}$,
\end{enumerate}
and if 
$$
\int_0 \frac{\gamma(\theta)}{\theta^2}\,d\theta=+\infty,
$$
then $|f(z)|\le 1$, $z\in\Omega_{\Lambda,E}$.
\label{x41}
\end{coro}

Later on, we need the following result.

\begin{proposition}
Let $\Lambda$ be a positive decreasing differentiable function  
on $(0,1]$ satisfying \eqref{cond1}, and let
\begin{equation}
\int_0 \frac{\gamma(\theta)}{\theta^2}\,d\theta<+\infty.
\label{x42}
\end{equation} 
There exists an outer function $F$ such that
\begin{equation}
|F(w)|>e^{\frac{1-|w|^{2}}{|1-w|^{2}}+\Lambda(\dist(w,E))},
\qquad w\in\partial \Omega_{\Lambda,E}\setminus\{1\}.
\label{x43}
\end{equation} 
\label{x44}
\end{proposition}

\begin{proof} 
By \eqref{x42}, we can set
$$
\log|F(e^{i\theta})|=A\frac{\gamma(\theta)}{\theta^2}
$$
for some $A$ to be chosen later.

Given $w=(1-\gamma(\theta))e^{i\theta}\in\partial \Omega_{\Lambda,E}$ and $\beta>0$ we have
$$
\log|F(w)|\ge \beta_1\min_{|\psi-\theta|<\beta\gamma(\theta)}
\frac{\gamma(\psi)}{\psi^2}
$$
with $\beta_1=\beta_1(\beta)>0$. Furthermore, by \eqref{condO}, 
for some $\beta>0$ we have
$$
\min_{|\psi-\theta|<\beta\gamma(\theta)}
\frac{\gamma(\psi)}{\psi^2}\ge \frac{\gamma(\theta)}{2\theta^2}.
$$
Now, \eqref{x43} follows for sufficiently large $A$.
\end{proof}

\section{General $E$}
\label{s6}

Let $\Lambda$ be a positive decreasing differentiable function  
on $(0,1]$ satisfying \eqref{cond1}, and let $E$ be a compact subset of $\mathbb T$, $1\in E$. We define $\gamma$ by \eqref{ga1}.

\begin{theo}
The function $S(z)=e^{-\frac{1+z}{1-z}}$ is cyclic in $\mathcal{B}_{\Lambda,E}^{\infty,0}$ if and only if the integral
\begin{equation}
\int \frac{\gamma(\theta)}{\theta^2}\,d\theta
\label{x46}
\end{equation}
diverges at $0$.
\label{teo0}
\end{theo}

\begin{proof} 
If the integral diverges, we use the same method as in the proof of Theorem~\ref{tem}; we use Corollary~\ref{x41} instead of Corollary~\ref{coro1} and Lemma~\ref{y11bis} 
instead of Lemma~\ref{y11}. The estimate 
\eqref{eq2} is replaced by \eqref{dop1}.

In the opposite direction, if the integral converges, we use 
Proposition~\ref{x44} and prove that $S$ is not cyclic by the Keldysh method (see \cite{KE}, \cite[Section 2.8.2]{NN}).
\end{proof} 

From now on we assume that
\begin{equation}
\Lambda(t)=\frac1{tw(t)^2},
\label{x47}
\end{equation}
where $w$ is a positive decreasing $C^1$ smooth function on $(0,1)$, $\lim_{t\to 0}w(t)=+\infty$, $w(t^2)\asymp w(t)$, 
$|w'(t)|=O(w(t)/t)$, $t\to 0$. Such $\Lambda$ satisfy \eqref{cond1}. 
Typical $w$ are $\log^p(1/t)$, 
$p>0$. 

\begin{remq}
For such $\Lambda$, 
the theorem of Nikolski \cite{NN} implies that if 
$\int_{0}\frac{dt}{|t|w(|t|)}<\infty$, then $S$ is not cyclic in 
$\mathcal{B}_{\Lambda,E}^{\infty,0}$; the theorem of 
Gevorkyan--Shamoyan \cite{GS} implies that if 
$\int_{0}\frac{dt}{|t|w(|t|)^2}=\infty$, then $S$ is cyclic in 
$\mathcal{B}_{\Lambda,E}^{\infty,0}$.
\label{rqk}
\end{remq}

Let $I_n$ be the arcs complementary to $E$, 
$I_n=(e^{ia_n},e^{ib_n})$ or $I_n=(e^{-ib_n},e^{-ia_n})$, $0<a_n<b_n$.
We divide the family of all such arcs into three groups:

the short intervals: $\mathcal I_1=\bigl\{I_n:1-\frac{a_n}{b_n}<\frac2{w(b_n)}\bigr\}$,

the intermediate intervals: $\mathcal I_2=\bigl\{I_n:\frac2{w(b_n)}\le 1-\frac{a_n}{b_n}<\frac12\bigr\}$,

and the long intervals: $\mathcal I_3=\bigl\{I_n:
\frac{a_n}{b_n}\le\frac12\bigr\}$.

\begin{theo} Let $\Lambda$ be defined by \eqref{x47} with 
$w$ satisfying the above conditions.
The function $S(z)=e^{-\frac{1+z}{1-z}}$ is cyclic in $\mathcal{B}_{\Lambda,E}^{\infty,0}$ if and only if 
\begin{multline}
\int_{e^{it}\in E\,\cup\,\bigcup_{I_n\in \mathcal I_1}I_n}\frac{dt}{|t|w(|t|)}+
\sum_{I_n\in \mathcal I_2}\frac1{w(b_n)^2}
\log\Bigl[\Bigl(1-\frac{a_n}{b_n}\Bigr)w(b_n)\Bigr]
\\+\sum_{I_n\in \mathcal I_3}\frac{\log w(b_n)}{w(b_n)^2} 
+\int_{e^{it}\in \bigcup_{I_n\in \mathcal I_3}I_n}\frac{dt}{|t|w^2(|t|)}=+\infty.
\label{x50}
\end{multline}
\label{teo1}
\end{theo}

\begin{remq}
One can easily verify that condition \eqref{x50} is equivalent 
to the divergence of at least one of the following three expressions: 
$$
\int_{e^{it}\in E}\frac{dt}{|t|w(|t|)},\quad  
\int_{0}\frac{dt}{|t|w^2(|t|)},\quad 
\sum_{n}\frac1{w(b_n)^2}
\log\Bigl[1+\Bigl(1-\frac{a_n}{b_n}\Bigr)w(b_n)\Bigr],
$$ 
where the sums runs by all the arcs $I_n$ complementary to $E$.
\end{remq}

\begin{proof} 
By Theorem~\ref{teo0}, we just need to study the convergence of the integral 
$$
\int\frac{\gamma(t)}{t^2}\,dt
$$
at $0$.

(a). For $e^{it}\in E$, $t>0$, we have
$$
\frac{\gamma(t)}{t^2}=\Lambda(\gamma(t))=\frac{1}{\gamma(t)w(\gamma(t))^2},
$$
and hence,
\begin{gather}
\gamma(t)w(\gamma(t))=t,\notag \\
\gamma(t)\asymp \frac{t}{w(t)},\notag\\
\frac{\gamma(t)}{t^2}\asymp\frac1{tw(t)}.\label{x51}
\end{gather}
Here we use that under our conditions on $w$, the function 
inverse to $t\mapsto tw(t)$ is equivalent to $t\mapsto t/w(t)$.

(b). Let $e^{it}\in I=\{e^{is}:0<a<s<b\}\in\mathcal I_1$.
(The case $b<t<a<0$ is treated analogously.) Then
$$
\frac{\gamma(t)}{t^2}=\Lambda(\gamma(t)+\dist(e^{it},E))
\le \Lambda(\gamma(t))
$$
and 
$$
\dist(e^{it},E)<|b-a|<\frac{2b}{w(b)}.
$$
Hence,
$$
\gamma(t)\lesssim \frac{t}{w(\gamma(t))}\lesssim \frac{t}{w(t)},
$$
and
\begin{gather*}
\Lambda(\gamma(t)+\dist(e^{it},E))\gtrsim \Lambda(\gamma(t)),\\
\gamma(t)\gtrsim \frac{t}{w(\gamma(t))}\gtrsim \frac{t}{w(t)}.
\end{gather*}
Finally,
\begin{equation}
\frac{\gamma(t)}{t^2}\asymp\frac1{tw(t)}.\label{x52}
\end{equation}

(c). Let $e^{it}\in I=\{e^{is}:0<a<s<b\}\in\mathcal I_2$.
We have 
$$
\frac{\gamma(t)}{t^2}\asymp
\frac1{(\gamma(t)+\dist(t,\{a,b\}))w(b)^2},
$$
and hence,
$$
\gamma(t)(\gamma(t)+\dist(t,\{a,b\}))\asymp \frac{b^2}{w(b)^2}.
$$
Therefore, for some $c>0$ and for 
$$
b-\frac{b}{cw(b)}<t<b
$$
we have
\begin{gather*}
\gamma(t)\asymp \frac{b}{w(b)},\\
\int_{b-\frac{b}{cw(b)}}^b \frac{\gamma(t)}{t^2}\,dt\asymp 
\frac{1}{w(b)^2},
\end{gather*}
and for
$$
\frac{a+b}{2}<t<b-\frac{b}{cw(b)}
$$
we have 
\begin{gather*}
\gamma(t)\asymp \frac{b^2}{w(b)^2\dist(t,\{a,b\})},\\
\int_{(a+b)/2}^{b-\frac{b}{cw(b)}} \frac{\gamma(t)}{t^2}\,dt\asymp 
\frac{1}{w(b)^2}\log\Bigl[c\Bigl(1-\frac{a}{b}\Bigr)w(b)\Bigr].
\end{gather*}
The integral from $a$ to $(a+b)/2$ is estimated in an analogous way, and we get 
\begin{equation}
\int_a^b\frac{\gamma(t)}{t^2}\asymp
\frac{1}{w(b)^2}\log\Bigl[\Bigl(1-\frac{a}{b}\Bigr)w(b)\Bigr].
\label{x53}
\end{equation}

(d). Let $I=\{e^{is}:0<a<s<b\}\in\mathcal I_3$.
As in part (c), the integral 
$$
\int_{(a+b)/2}^b\frac{\gamma(t)}{t^2}
$$
is equivalent to 
$$
\frac{\log w(b)}{w(b)^2}.
$$
Next,
$$
\int_{a}^{a+\frac{a}{w(a)}}\frac{\gamma(t)}{t^2}
\asymp \frac{1}{w(a)^2}.
$$
For $t\in (a+\frac{a}{w(a)},\frac{a+b}2)$ we have
$$
\frac{\gamma(t)}{t^2}=\Lambda(\gamma(t)+(t-a))\le \Lambda(\gamma(t)),
$$
and
$$
\gamma(t)\le \frac{t}{w(\gamma(t))}\lesssim \frac{t}{w(t)}.
$$
Therefore,
$$
\frac{\gamma(t)}{t^2}\asymp \Lambda(t-a),
$$
and
$$
\int_{a+\frac{a}{w(a)}}^{(a+b)/2}\frac{\gamma(t)}{t^2}\,dt
\asymp \int_{\frac{a}{w(a)}}^{(b-a)/2}\frac{dt}{tw(t)^2}\asymp
\int_a^b \frac{dt}{tw(t)^2}+\frac{\log w(a)}{w(a)^2}.
$$
Thus,
\begin{equation}
\int_{a}^{b}\frac{\gamma(t)}{t^2}\,dt\asymp
\int_a^b \frac{dt}{tw(t)^2}+\frac{\log w(b)}{w(b)^2}.
\label{x54}
\end{equation}
The theorem follows from \eqref{x51}--\eqref{x54}.
\end{proof}

\begin{coro}Let $\Lambda$ be as in the formulation of Theorem~\ref{teo1}. 
This theorem yields immediately that if $E=\{\exp(i\cdot2^{-n})\}_{n\ge 1}\cup\{1\}$, then $S$ is cyclic in 
$\mathcal{B}_{\Lambda,E}^{\infty,0}$ if and only if 
$$
\sum_{n\ge 1}\frac{\log w(2^{-n})}{w(2^{-n})^2}=+\infty\,;
$$
if $E=\{\exp(i\cdot2^{-2^n})\}_{n\ge 1}\cup\{1\}$, then $S$ is cyclic in 
$\mathcal{B}_{\Lambda,E}^{\infty,0}$ if and only if 
$$
\int_0 \frac{dt}{tw(t)^2}=+\infty,
$$
and we return to the Gevorkyan--Shamoyan condition valid for 
$E=\{1\}$.
\end{coro}

Next we give two more applications of the general criterion 
\eqref{x50}. 

Let us introduce a condition 
\begin{equation}
\tag{$C_\beta$}  \int_0\frac{\Lambda(t)^{1-\beta}}{t^\beta}\,dt=+\infty,\qquad 0\le \beta\le\frac12
\end{equation}
interpolating between that by Nikolski ($\int_0\sqrt{\Lambda(t)/t}\,dt=+\infty$, $\beta=1/2$) and that by 
Gevorkyan--Shamoyan ($\int_0\Lambda(t)\,dt=+\infty$, $\beta=0$).

For 
$$
\Lambda_\alpha(t)=\frac{1}{t\log^\alpha(1/t)}
$$
we have
$$
\Lambda_\alpha\in(C_\beta)\iff \alpha(1-\beta)\le 1.
$$

\begin{theo} Let $0\le \beta\le 1/2$, $a_n=\exp(-n^{1-\beta})$, 
$n\ge 1$, $E_\beta=\{e^{ia_n}\}_{n\ge 1}\cup\{1\}$.
The function $S(z)=e^{-\frac{1+z}{1-z}}$ is cyclic in 
$\mathcal{B}_{\Lambda_\alpha,E_\beta}^{\infty,0}$ if and only if 
$\Lambda_\alpha\in(C_\beta)$.
\label{teo2}
\end{theo}

\begin{proof} We have 
$$
1-\frac{a_{n+1}}{a_n}=1-\exp\bigl[n^{1-\beta}-(n+1)^{1-\beta}\bigr]\asymp n^{-\beta}.
$$

Consider three cases.

(a). $\alpha(1-\beta)>1$. Then all the arcs 
$\{e^{it}\}_{a_{n+1}<t<a_n}$ are intermediate ones, and we need only to verify that 
$$
\sum_{n\ge 1}\frac{1}{\log^\alpha(\exp(n^{1-\beta}))}
\log\Bigl[n^{-\beta}\log^{\alpha/2}(\exp(n^{1-\beta}))\Bigr]
\asymp \sum_{n\ge 1}\frac{\log(n^{-\beta+\alpha(1-\beta)/2})}
{n^{\alpha(1-\beta)}}<+\infty.
$$

(b). $2\beta<\alpha(1-\beta)\le 1$. Again all the arcs 
$\{e^{it}\}_{a_{n+1}<t<a_n}$ are intermediate ones, and 
$$
\sum_{n\ge 1}\frac{1}{\log^\alpha(\exp(n^{1-\beta}))}
\log\Bigl[n^{-\beta}\log^{\alpha/2}(\exp(n^{1-\beta}))\Bigr]
\asymp \sum_{n\ge 1}\frac{\log(n^{-\beta+\alpha(1-\beta)/2})}
{n^{\alpha(1-\beta)}}=+\infty.
$$

(c). $\alpha(1-\beta)\le 2\beta\le 1$. In this case we can assume that all the arcs are short ones, and we have
$$
\int_0\frac{dt}{t\log^{\alpha/2}(1/t)}=+\infty.
$$
Together, (a), (b), and (c) prove the assertion of the theorem.
\end{proof}

Finally, we deal with the Cantor ternary set $F$. Let $F_0=[0,1]$. 
On step $n\ge 0$, $F_n$ consists of $2^n$ intervals 
$I_j=[a_j,b_j]$. We divide each of them into three equal 
subintervals 
$$
I_j=\bigl[a_j,\frac{2a_j+b_j}3\bigr]\cup 
\bigl[\frac{2a_j+b_j}3,\frac{a_j+2b_j}3\bigr]\cup 
\bigl[\frac{a_j+2b_j}3,b_j\bigr]
$$ 
and set
$$
F_{n+1}=\bigcup_{j}I^1_j\cup\bigcup_{j}I^3_j.
$$
We define $F=\cap_{n\ge 1}F_n$. Denote by $\kappa$ 
the Hausdorff dimension of $F$ (see \cite[Section 1.5]{FA}), 
$\kappa=\frac{\log 2}{\log 3}$.

\begin{theo} Let $E=\{e^{it}:t\in F\}$. 
The function $S(z)=e^{-\frac{1+z}{1-z}}$ is cyclic in 
$\mathcal{B}_{\Lambda_\alpha,E}^{\infty,0}$ if and only if 
$$
\alpha\le \frac{1}{1-\frac{\kappa}{2}}.
$$
\label{teo3}
\end{theo}

\begin{proof} The set $E$ is of zero measure; all the complementary arcs are short or intermediate. For simplicity, 
we pass to $F\subset[0,1]$. For every $N\ge 1$ we have $2^N$ 
complementary intervals of length $3^{-N}$.
In every interval $[3^{s-N},2\cdot 3^{s-N}]$ we have $2^s$ of such intervals, $0\le s<N$. They are short for 
$3^{-s}(N-s)^{\alpha/2}\lesssim 1$ and intermediate for 
$3^{-s}(N-s)^{\alpha/2}\gtrsim 1$.

The sum for the intermediate intervals in \eqref{x50} is
\begin{multline*}
\sum_{N\ge 1}\,\,\sum_{s\ge 0,\,3^{-s}(N-s)^{\alpha/2}\gtrsim 1}
\,\,
\sum_{I_n=[a_n,b_n]\subset [3^{s-N},2\cdot 3^{s-N}]}
\log^\alpha\frac{1}{b_n}\log\Bigl[\frac{b_n-a_n}{b_n}
\log^{\alpha/2}\frac{1}{b_n}\Bigr]
\\ \asymp 
\sum_{N\ge 1}\,\,\sum_{s\ge 0,\,3^{-s}(N-s)^{\alpha/2}\gtrsim 1}
\frac{\log^+(3^{-s}(N-s)^{\alpha/2})}{(N-s)^\alpha}\cdot 2^s\asymp\sum_{N\ge 1}\frac{N^{\alpha\kappa/2}}{N^\alpha};
\end{multline*}
the latter series diverges if and only if 
$\alpha(1-\frac{\kappa}2)\le 1$.

The integral for the short intervals in \eqref{x50} is
\begin{multline*}
\sum_{N\ge 1}\,\,\sum_{s\ge 0,\,3^{-s}(N-s)^{\alpha/2}\lesssim 1}
\,\,
\sum_{I_n=[a_n,b_n]\subset [3^{s-N},2\cdot 3^{s-N}]}
\int_{a_n}^{b_n}\frac{dt}{tw(t)}
\\ \asymp \sum_{N\ge 1}\,\,\sum_{s\ge 0,\,3^{-s}(N-s)^{\alpha/2}\lesssim 1}
\,\,
\sum_{I_n=[a_n,b_n]\subset [3^{s-N},2\cdot 3^{s-N}]}
\frac1{(N-s)^{\alpha/2}}\log\frac{b_n}{a_n}
\\ \asymp \sum_{N\ge 1}\,\,\sum_{s\ge 0,\,3^{-s}(N-s)^{\alpha/2}\lesssim 1}
\frac1{N^{\alpha/2}}\cdot\frac{2^s}{3^s}
\asymp 
\sum_{N\ge 1}\frac{N^{(\alpha/2)(\kappa-1)}}{N^{\alpha/2}};
\end{multline*}
the latter series diverges if and only if 
$\alpha(1-\frac{\kappa}2)\le 1$.
\end{proof}

Since the Cantor set $F$ is self-similar, we get the same result for every shift of $E$: $E_x=\{e^{i(y-x)}:e^{iy}\in E\}$, $e^{ix}\in E$.
On the other hand, it looks difficult to characterize the threshold value of $\alpha$ in terms of (the local behavior near the point $1$) for general sets $E$.

\end{document}